\documentclass[11 pt,a4paper,twoside,reqno]{amsart}
\usepackage{amsfonts,amssymb,amscd,amsmath,enumerate,verbatim,calc}

\usepackage[applemac]{inputenc}
\usepackage[pdftex]{graphicx}

 \usepackage{hyperref}

\usepackage{xcolor}
	\hypersetup{
    colorlinks,
    linkcolor={red!50!black},
    citecolor={blue!50!black},
    urlcolor={blue!80!black}
}

\usepackage{xy}
\usepackage{xypic} 
 
\usepackage{multicol}
\usepackage{xcolor}
\usepackage{amsthm} 
\usepackage{mathrsfs} 
\def\C{{\mathbb C}}
\def\N{{\mathbb N}}

\def\Z{{\mathbb Z}}

\def\calT{{\mathcal T}}

\def\rmd{{\mathrm d}}
\def\rme{{\mathrm e}}
\def\rmi{{\mathrm i}}

\def\uk{{\underline{k}}}
\def\uz{{\underline{z}}}
\def\uw{{\underline{w}}}
\def\ut{{\underline{t}}}
\def\us{{\underline{s}}}
\def\ue{{\underline{e}}}
\def\utau{{\underline{\tau}}}
\def\utheta{{\underline{\theta}}}
\def\uw{{\underline{w}}}
\def\unu{{\underline{\nu}}}

\newtheorem{theorem}{Theorem}
\newtheorem{corollary}{Corollary}
\newtheorem{proposition}{Proposition}
\newtheorem{lemma}{Lemma}
\newtheorem*{remark}{Remark}

\begin{document}

\noindent{\small 
\href{http://www.rsmams.org/}{Journal of Ramanujan Society of Mathematics and Mathematical Sciences (JRSMAMS)}
\\
Vol. 9, No. 1, December 2021
   }
 
 \bigskip
\begin{center}
{
\Large
\bf 
On entire functions of several variables 
\\
with derivatives of even order taking integer values

\bigskip
\large
 Michel Waldschmidt
 
 \hskip .01 true cm
 \vtop{\hsize 1cm \kern 1mm
 \hrule  width 16mm  \kern 1mm}
 }

 \sl
 Sorbonne Université and Université de Paris, CNRS, IMJ-PRG, F-75005 Paris, France.

 \tt
 \href{mailto:michel.waldschmidt@imj-prg.fr}{michel.waldschmidt@imj-prg.fr}
 
 \url{http://www.imj-prg.fr/~michel.waldschmidt}

\end{center}

\section*{Abstract} 

We extend to several variables an earlier result of ours, according to which an entire function of one variable of sufficiently small exponential type, having all derivatives of even order  taking integer values at two points, is a polynomial. The proof in the one dimensional case relies on Lidstone expansion of the function. For $n$ variables, we need $n+1$ points, having the property that the differences of $n$ of them with the remaining one give a basis of $\C^n$. The proof is by reduction to the one variable situation.

\bigskip
\subsection*{Keywords} 
 Integer valued entire functions, Lidstone polynomials, exponential type, P\'olya's Theorem, analytic functions of several variables.

\bigskip
\noindent
{\bf AMS Mathematics Subject Classification 2020}:
30D15, 41A58.
 
 \section{The main result}

We denote by $\N$ the set $\{0,1,2,\dots\}$. 
For $\uz=(z_1,\dots,z_n)\in\C^n$ and $\ut=(t_1,\dots,t_n)\in\N^n$, write 
$$
\uz^\ut=z_1^{t_1}\cdots z_n^{t_n},\quad |\uz|=\max_{1\le i\le n}|z_i|, \quad\Vert \ut\Vert=t_1+\cdots+t_n, \quad \ut!=t_1!\cdots t_n!
$$
and
$$
D^\ut=\left(\frac{\partial}{\partial z_1}\right)^{t_1}
\cdots
\left(\frac{\partial}{\partial z_n}\right)^{t_n}.
$$
For $\uz=(z_1,\dots,z_n)$ and $\uw=(w_1,\dots,w_n)$ in $\C^n$, let $\uw\uz=w_1z_1+\cdots+w_nz_n$. 
For $r>0$ and for $f$ an analytic function in a domain containing  $\{\uz\in\C^n\; \mid \; |\uz|\le r\}$, set
$$
|f|_r=\sup_{|\uz|=r}|f(\uz)|.
$$
The order of an entire function  $f$  is
$$
\varrho(f)=\limsup_{r\to\infty} \frac{\log\log |f|_r}{\log r}
$$
and its exponential type 
$$
\tau(f)= \limsup_{r\to\infty} \frac{ \log |f|_r}{r}\cdotp
$$
Given $\tau>0$, $\uw\in\C^n\setminus\{0\}$ and an entire function $f$ in $\C^n$, we say that $f$ has \emph{exponential type $\le \tau$ in the direction $\uw$} if the function of a single variable $\uz\mapsto f(\uw\uz )$ has order $\le 1$ and exponential type $\le \tau$, namely
$$
\limsup_{r\to\infty} \frac{1}{ r}  \log \sup_{|\uz|\le r} |f(\uw\uz)|\le \tau.
$$
It depends not only on $\uw/|\uw|$ but also on $|\uw|$. 

 Let $\us_0,\us_1,\dots,\us_n$ be $n+1$ elements in  $\C^n$, such that $\{\us_1-\us_0,\dots,\us_n-\us_0\}$ is a basis of $\C^n$ over $\C$. 
 
Finally, let $\calT$ be the set of  $(\ut,i)\in\N^n\times\{0,\dots,,n\}$ satisfying
$$
\Vert \ut \Vert \text{ even for all  $i=0,1,\dots,n$ and } t_1,\ldots,t_i\text{ even for }i=1,\dots,n.
$$
 
 The main result of this paper is the following.

\begin{theorem}\label{Th:but}
Let $f$ be an entire function of  $n$ variables having exponential type $\le \tau $  in each of the directions $\us_i-\us_0$ ($i=1,\dots,n$) satisfying
\begin{equation}\label{eq:maingrowthcondition}
\limsup_{r\to\infty}\rme^{-r}\sqrt r |f|_r<\frac{1}{\sqrt{2\pi}}
\rme^{-\max \{|s_0|,\dots,|s_n|\}}.
\end{equation}
Assume  
\begin{equation}\label{Eq:but}
(D^\ut f)(\us_i) \in\Z \text{ for all $(\ut,i)\in\calT$.}
\end{equation}
Then  the set of  $(\ut,i)\in\calT$ with $(D^\ut f)(\us_i)\not=0$ is finite.  
Further, if  
\begin{equation}\label{eq:majorationtype}
\tau<\pi,
\end{equation}
then  $f$  is a polynomial in $\C[\uz]$. 
\end{theorem}

The case  $n=1$ of Theorem \ref{Th:but} is  Corollary 1.2 of  \cite{SEAMS}. 

We will see (Corollary \ref{Corollary:coefficientsrationnels}) that when $K$ is a subfield of $\C$ such that $\us_i\in K^n$ for all $i=0,1,\dots,n$, then the conclusion of Theorem \ref{Th:but}  is $f\in K[\uz]$. 

Assumption \eqref{eq:maingrowthcondition} implies that $f$ has exponential type $\le 1$. Hence, when 
$$
\max_{1\le i\le n}|\us_i-\us_0|< \pi,  
$$
the hypothesis \eqref{eq:majorationtype} is a consequence of \eqref{eq:maingrowthcondition}. An example is  when $\us_0=(0,\dots,0)$ and $(\us_1,\dots,\us_n)$ is the canonical basis  of $\C^n$.

Theorem \ref{Th:but} cannot be improved in general. Here are some examples. Assume 
 $$
 \us_0=(a_1,a_2,\dots,a_n), \;    \us_i=\us_0+(b_i-a_i)\ue_i \quad (i=1,\dots,n),
 $$
where $a_i$ and $b_i$ are complex numbers with $a_i\not=b_i$ for $1\le i\le n$, while $\ue_1,\dots,\ue_n$ is the canonical basis of $\C^n$.  
Our first example is the function 
$$
\sin\left(\pi\frac{z_1-a_1}{ b_1-a_1}+\cdots +  \pi\frac{z_n-a_n}{ b_n-a_n}\right),
$$
which has exponential type $\le \pi$ in each of the directions $\us_i-\us_0$ ($i=1,\dots,n$) and satisfies  $(D^\ut f)(\us_i)=0$ for $i=0,1,\dots,n$ and  for any $\ut\in\N^n$ with $\Vert \ut\Vert$ even.

For our second  example, we define $\utheta:\C^n\to\C^n$ as follows
$$
\utheta(\uz)=\left(
\frac{z_1-a_1}{b_1-a_1},\dots,\frac{z_n-a_n}{b_n-a_n}\right). 
$$
 For $1\le i\le n-1$, let $g_i$ be a polynomial in $n-i$ variables with complex coefficients, and let $g_n$ be a polynomial in a single variable with complex coefficients.
 Consider the entire function of $n$ variables $\uw=(w_1,\dots,w_n)$:
 $$
 \varphi(\uw)=\sum_{i=1}^{n-1} \sin(\pi w_i)g_i(w_{i+1}^2,\dots,w_n^2)+\sin(\pi w_n)g_n(w_{n-1}^2).
$$ 
 Then the  function
 $$
 f(\uz)=\varphi\circ \utheta(\uz)
$$
has exponential type $\le \pi$ in each of the directions $\us_i-\us_0$ ($i=1,\dots,n$) and satisfies  $(D^\ut f)(\us_i)=0$ for all $(\ut,i)\in\calT$. 
 
For our third example, we keep the same notation for  $\us_i$ and $\utheta$, we assume that $a_i-b_i\not\in\pi\rmi\Z$ for $1\le i\le n$, we also assume that the polynomials $g_i$ have integer coefficients and  we set
$$
 \psi(\uw)=\sum_{i=1}^{n-1} \frac{\sinh(w_i- b_i)}{\sinh(a_i-b_i)} g_i(w_{i+1}^2,\dots,w_n^2)+\frac{\sinh(w_n- b_n)}{\sinh(a_n-b_n)} g_n(w_{n-1}^2).
$$
Then  the function 
 $$
 f(\uz)=
 \psi\circ \utheta(\uz)
$$
has exponential type $\le 1$ in each of the directions $\us_i-\us_0$ ($i=1,\dots,n$) and satisfies  $(D^\ut f)(\us_i)\in\Z$ for all $(\ut,i)\in\calT$. 
 
Another reason for which Theorem \ref{Th:but} is optimal is the following. If we relax the assumption  \eqref{Eq:but} by requiring that it holds outside a finite subset of $\calT$, the conclusion that $f$ is a polynomial is still valid - this follows easily from Theorem \ref{Th:but}. But if we impose  the conditions   \eqref{Eq:but} only outside an infinite subset of $\calT$, then the conclusion is no more valid. We come back to this issue in the remark at the end of \S~\ref{S:SpecialCase}.

\section{An extension of a result due to P\'olya}

The proof of the first part of Theorem \ref{Th:but} rests on Proposition \ref{Proposition:Polya}, which is an extension to several variables of  Proposition 2.2 of  \cite{SEAMS}.

We will use Cauchy's inequalities for an analytic function of several variables  (\cite[Theorem 2.2.7 p. 27]{Hormander}). 
Let  $r>0$, let $\ut\in\N^n$ with $\Vert \ut\Vert =T$ and let $f$ be analytic in a domain containing  $\{\uz\in\C^n\; \mid \; |\uz|\le r\}$. Then 
$$
\frac{|(D^\ut f)(0)|}{\ut!}r^T \le |f|_{r}.
$$
We deduce that for  $\uz_0\in\C^n$ and for   $f$ analytic in a domain containing  
$$
\{\uz\in\C^n\; \mid \; |\uz|\le r+|\uz_0|\}, 
$$
we have
\begin{equation}\label{Equation:CauchyInequality}
\frac{|(D^\ut f)(\uz_0)|}{\ut!}r^T \le |f|_{r+|\uz_0|}.
\end{equation}

We will also use Stirling's Formula:
\begin{equation}\label{Equation:Stirling}
N^N\rme^{-N}\sqrt{2\pi N}
< N! <
N^N \rme^{-N} \sqrt{2\pi N} \rme^{1/(12N)}, 	
\end{equation}
which is valid for all $N\ge 1$.

\begin{proposition}\label{Proposition:Polya}
Let $f$ be an entire function in $\C^n$ and let $A\ge 0$. Assume 
\begin{equation}\label{eq:maingrowthconditionA}
\limsup_{r\to\infty}\rme^{-r}\sqrt r |f|_r< \frac{ \rme^{-A}}{\sqrt{2\pi} }\cdotp
\end{equation}
Then there exists $T_0>0$ such that, for $\ut\in\N^n$ with $\Vert \ut\Vert\ge T_0$, we have 
$$
|D^\ut f|_A<1.
$$ 
\end{proposition}

\begin{proof}
From assumption \eqref{eq:maingrowthconditionA}, it follows that there exists  $\eta>0$ such that, for $T$ sufficiently large, we have 
$$
|f|_T< (1-\eta) \frac{ \rme^{T-A}}{\sqrt{2\pi T} }\cdotp
$$
Let  $\ut\in\N^n$ with  $\Vert \ut\Vert=T$.
We use Cauchy's inequalities \eqref{Equation:CauchyInequality} with $r=T-A$: since $\ut!\le T!$, for $|\uz|\le A$ we have 
$$
|(D^\ut f)(\uz)|\le \frac{T!}{(T-A)^T}|f|_T.
$$
Hence the upper bound for $T!$ given by the right hand side of  \eqref{Equation:Stirling} yields
$$
\sup_{|\uz|\le A} |(D^\ut f)(\uz)| \le (1-\eta) 
 \rme^{-A+1/(12T)} 
 \left( 1-\frac A T\right)^{-T}.
$$
For $T$ sufficiently large, the right hand side is $<1$.
\end{proof}

We deduce the following extension to several variables of Corollary 2.4 of  \cite{SEAMS}.

\begin{corollary}\label{Corollary:Polya}
Let $f$ be  a transcendental 
entire function in $\C^n$. Let $A\ge 0$. 
Assume 
\eqref{eq:maingrowthconditionA}. 
Then the set
$$
\bigl\{(\ut,\uz_0)\in\N^n\times \C^n \; \mid \; |\uz_0|\le A, \; (D^\ut f) (\uz_0)\in\Z\setminus\{0\} \bigr\}
$$
is finite. 
\end{corollary}

\section{A special case: $(\us_0,\us_1,\dots,\us_n)=(\ue_0,\ue_1,\dots,\ue_n)$}\label{S:SpecialCase}

Set $\ue_0=(0,\dots,0)\in\C^n$ and denote, as before,  by 
$\{\ue_1,\dots,\ue_n\}$ the canonical basis   of $\C^n$: 
$$
\ue_{ij}=\delta_{ij} \quad (1\le i,j\le n).
$$
We will say that an entire function $f$ in $\C^n$ has \emph{exponential type $\le \tau $  in each of the variables} if it has  exponential type $\le \tau $  in each of the directions  $\ue_1,\dots,\ue_n$: in other words, for any $i=1,\dots,n$ and any $(z_1,\dots,z_{i-1},z_{i+1},\dots,z_n)\in\C^{n-1}$, 
$$
\limsup_{r\to\infty} \frac{1}{ r} \log \sup_{|z_i|\le r} |f(z_1,\dots,z_n)|\le \tau.
$$
The main result of this section is the following.

\begin{proposition} \label{Prop:PoritskyGeneralise}
Let   $f$  be an entire function in $\C^n$ of exponential type $<\pi$  in each of the variables. Assume 
\begin{equation}\label{Eq:unicityei}
(D^\ut f)(\ue_i)=0\text{ for all } (\ut,i)\in\calT.
\end{equation}
Then $f=0$. 
\end{proposition}

The case  $n=1$ of Proposition \ref{Prop:PoritskyGeneralise}  is a result due to Poritsky  (cf. \cite{Poritsky} and \cite[\S~3.1]{SEAMS}). We will prove  Proposition \ref{Prop:PoritskyGeneralise} by induction on $n$, starting with and including the case $n=1$. The proof in the one dimensional case will use the following two well known lemmas dealing with univariate functions. 

\begin{lemma}\label{Lemma:f(z)periodic}
Let $f$ be an entire function in $\C$. The two following conditions are equivalent:
\\
{\rm (i)} The  function $f$ is periodic of period $\omega\not=0$ ;
\\
{\rm (ii)} There exists a function $g$ analytic in $\C^\times$ such that $f(z)=g(\rme^{2\pi\rmi  z/\omega})$. 
\end{lemma}

\begin{proof}
Clearly {\rm (ii) } implies {\rm (i)}. Assume {\rm (i)}. 
The map  $z\mapsto \rme^{\pi\rmi z}$ is analytic and surjective. 
The condition $\rme^{\pi\rmi  z_1}=\rme^{\pi\rmi  z_2}$ implies $f(z_1)=f(z_2)$. 
Hence there exists a unique map $g:\C^\times\to\C$ such that  $g(\rme^{2\pi\rmi  z})=f(z)$. 
$$\xymatrix{ \C \ar[d]^{^{\rme^{2\pi\rmi  z}}} \ar[r]^f & \C \cr
 \C^\times \ar@{.>}[ur]_g}
$$
Let  $w\in\C^\times$ and let $z\in\C$  be such that  $w=\rme^{2\pi\rmi  z}$. From $g(w)=f(z)$ it follows that $g$ is holomorphic, hence analytic, in $\C^\times$.

This proves  lemma \ref{Lemma:f(z)periodic}. 
\end{proof}

\begin{lemma}\label{lemma:g(exp2ipiz)}
If $g$ is an analytic function in $\C^\times$, if $\omega$ is a nonzero complex number  and if the entire function $g(\rme^{2\pi\rmi  z/\omega})$ has an exponential  type $< 2(N+1)\pi/|\omega|$ for some nonnegative integer $N$, then $w^Ng(w)$ is a polynomial of degree $\le 2N$. 

As a consequence, if   $g(\rme^{2\pi\rmi  z/\omega})$ has a type $< 2\pi/|\omega|$, then $g$ is constant.
\end{lemma}

\begin{proof}
 Assume that the function $f(z)= g(\rme^{2\pi\rmi  z/\omega})$ has an exponential type $\tau$ with $\tau< 2(N+1)\pi/|\omega|$.
Let $w\in\C^\times$. Write $w=|w|\rme^{i\theta}$ with $|\theta|\le \pi$. Set
$$ 
z=\frac{\omega} {2\pi\rmi }(\log |w|+i\theta),
$$
so that $w=\rme^{2\pi\rmi  z/\omega}$. For any $\epsilon_1>0$,  we have
$$ 
|z|\le \left(\frac{\omega} {2\pi} +\epsilon_1\right) | \log |w| | 
$$
for sufficiently large $|w|$  and also for sufficiently small $|w|$. We deduce 
$$
\log|g(w)|=\log |f(z)|\le (\tau+\epsilon_2)|z|\le
\left(\frac{\omega \tau}{2\pi} +\epsilon_3\right) |\log |w| |.
$$
Hence if  $\alpha$ satisfy $\frac {\tau |\omega|}{2\pi}< \alpha<N+1$,
then  $|g|_r\le r^{\alpha}$ for sufficiently large $r$  and
$|g|_r\le r^{-\alpha}$ 
for sufficiently small $r>0$ . 
Consider the Laurent expansion of $g$ at the origin: 
$$
g(w)=\sum_{n\in\Z} b_nw^n.
$$
From
$$
b_n =\frac 1 {2\pi} \int_{|w|=r} g(w) \frac {\rmd w}{w^{n+1}}
$$
we deduce Cauchy's inequalities 
$$
|b_n|r^n\le \frac 1 {2\pi}|g|_r.
$$
For $n>N$, we use these inequalities with $r\to \infty$ while for $n<-N$, we use these inequalities with $r\to 0$.
We deduce $b_n=0$ for $|n|\ge N+1$. Hence 
$$
g(w)=\frac 1 {w^N}A(w)+B(w)
$$
where $A$ and $B$ are polynomials of degree $\le N$. 
\end{proof}

\begin{proof}[Proof of Proposition \ref{Prop:PoritskyGeneralise}]
We start by proving the case  $n=1$ of Proposition \ref{Prop:PoritskyGeneralise},  due to Poritsky. 
So let $f$ be an entire function of a single variable of exponential type $<\pi$ satisfying $f^{(t)}(0)=f^{(t)}(1)=0$ for all  even $t\ge 0$.
We claim that this implies $f=0$. 

Indeed, from the assumptions it follows that the functions $f(z)$ and $f(1-z)$ are odd, hence $f(z)$ is periodic of period $2$. Lemma \ref{Lemma:f(z)periodic}  gives the existence of an entire function $g$ such that $f(z)=g(\rme^{\pi\rmi  z})$. Since $f(z)$ has exponential type $<\pi$,  Lemma  \ref{lemma:g(exp2ipiz)} implies that $g$ is a constant, hence $f$ also. From $f(0)=0$ we conclude $f=0$. 

We now prove  Proposition \ref{Prop:PoritskyGeneralise} by induction on the number $n$ of variables. Let $n\ge 2$ and let 
$$
f(\uz)=\sum_{\uk\in\N^n} a_\uk \uz^\uk
$$
be an entire function of $n$ variables of exponential type $<\pi$ in each of the variables satisfying 
$$
(D^\ut f)(\ue_i)=0 \text{ for all  $(\ut,i)\in\calT$.}
$$
For $k_n\ge 0$, define an entire function $f_{k_n}$ of $n-1$ variables, having  exponential type $<\pi$ in each of the $n-1$ variables, by setting 
$$
f_{k_n}(z_1,\dots,z_{n-1})=\sum_{(k_1,\dots,k_{n-1})\in \N^{n-1}} a_\uk z_1^{k_1}\cdots z_{n-1}^{k_{n-1}}
=k_n!\left(\frac{\partial}{\partial z_n}\right)^{k_n} f(z_1,\dots,z_{n-1},0),
$$
so that 
$$
f(\uz)=\sum_{k_n\ge 0}f_{k_n}(z_1,\dots,z_{n-1})z_n^{k_n}.
$$
Let $k_n\ge 0$ be even. 
For each $(t_1,\dots,t_{n-1})\in\N^{n-1}$, we have
$$ 
\left(\frac{\partial}{\partial z_1}\right)^{t_1}\cdots  \left(\frac{\partial}{\partial z_{n-1}}\right)^{t_{n-1}} f_{k_n}
(z_1,\dots,z_{n-1})=k_n!(D^{t_1,\dots,t_{n-1},k_n}f)(z_1,\dots,z_{n-1},0).
$$
If $((t_1,\dots,t_{n-1}),i)\in\N^{n-1} \times\{0,1,\dots,n-1\}$  is such that $t_1+\cdots+t_{n-1}$ is even and $t_1,\dots,t_i$ are even, then $((t_1,\dots,t_{n-1},k_n),i)\in\calT$. From the assumption we deduce
$$
(D^{t_1,\dots,t_{n-1},k_n}f)(\ue_i)=0.
$$
Using the induction hypothesis for  $n-1$ variables, we deduce $f_{k_n}=0$ for all $\uk\in\N^n$ with  $k_n$ even, hence $a_\uk =0$ for all $\uk\in\N^n$ with  $k_n$ even. 

Since  $\ue_{n-1}$ and $\ue_n$ play the same role, we also have $ a_\uk =0$ for all $\uk\in\N^n$ with  $k_{n-1}$ even. Therefore the condition $ a_\uk \not=0$ implies that $k_{n-1}$ and $k_{n}$ are both odd, and this implies that $k_{n-1}+k_n$ is even.

 We now complete the proof of Proposition \ref{Prop:PoritskyGeneralise} 
 in the case $n=2$: the hypothesis 
 $$
 (D^{k_1,k_2}f)(0,0)=0\text { for all ($k_1,k_2)\in\N^2$ with $k_1+k_2$ even }
 $$
 implies 
$a_{k_1,k_2}=0 $ for all ($k_1,k_2)\in\N^2$ with $k_1$ and $k_2$ both odd, hence, using what we already proved,  $a_{k_1,k_2}=0 $ for all ($k_1,k_2)\in\N^2$, and therefore $f=0$.  

Finally, assume $n\ge 3$. Let us fix  $k_{n-1}$ and $k_{n}$, both odd, and consider the entire function of  $n-2$ variables 
$$
\begin{aligned}
f_{k_{n-1},k_n}(z_1,\dots,z_{n-2})
&=\sum_{(k_1,\dots,k_{n-2})\in \N^{n-2}} a_\uk z_1^{k_1}\cdots z_{n-2}^{k_{n-2}}
\\
&=k_{n-1}!k_n!\left(\left(\frac{\partial}{\partial z_{n-1}}\right)^{k_{n-1}} \left(\frac{\partial}{\partial z_n}\right)^{k_n} f\right)(z_1,\dots,z_{n-2},0,0).
\end{aligned}
$$
If $t_1+\cdots+t_{n-2}$ is even, if $i$ satisfies $0\le i\le n-2$  and if $t_1,\dots,t_i$ are even, then $((t_1,\dots,t_{n-2},k_{n-1},k_n),i)\in\calT$. 
From the induction hypothesis with   $n-2$ variables, we deduce that this function $f_{k_{n-1},k_n}$  is $0$. Hence $a_\uk=0$ for all $\uk\in\N^n$, and finally  $f=0$. 
 \end{proof}

 \noindent
 \begin{remark}{\rm
 Using Proposition \ref{Prop:PoritskyGeneralise}, one can prove that there exists a unique family of polynomials $\Lambda_{\ut,i}\in\C[\uz]$ ($(\ut,i)\in\calT$) which satisfy, for all $(\utau,j)\in\calT$ and  $(\ut,i)\in\calT$, 
$$
(D^\utau \Lambda_{\ut,i})(\ue_j)=\delta_{\ut,\utau}\delta_{ij}.
$$
These polynomials  generalize Lidstone polynomials to several variables. 
In a forthcoming paper \cite{MultivariateLidstone}, we study these polynomials and we prove that any entire function $f$ in $\C^n$ of exponential type $<\pi$ in each variable is the sum of a series 
$$
f(\uz)=\sum_{(\ut,i)\in\calT} (D^\ut f)(\ue_i)\Lambda_{\ut,i}(\uz).
$$
This generalizes a result of Poritsky  (cf. \cite{Poritsky} and \cite[\S~3.1]{SEAMS}) for univariate entire functions. 

In \cite{MultivariateLidstone}, we also show that if $\calT'$ is a subset of $\calT$ such that $\calT\setminus\calT'$ is infinite, then there exists   
an uncountable set of transcendental entire functions $f$  of exponential type $0$ such that $(D^\ut f)(\ue_i)=0$ for all $(\ut,i)\in\calT'$.

}
\end{remark} 

\section{Change of coordinates}
 
 We deduce from Proposition \ref{Prop:PoritskyGeneralise} the following result

\begin{proposition}\label{Prop:PoritskyCasGeneral}
An entire function $f$ in $\C^n$ of exponential type $<\pi$ in each of the directions $\us_i-\us_0$ ($i=1,\dots,n$) which satisfies 
\begin{equation}\label{Eq:unicityCasGeneral}
(D^\ut f)(\us_i)=0
\end{equation}
for all $(\ut,i)\in\calT$ is the zero function. 
\end{proposition}

\begin{proof}
Set 
$$
\tilde{f}(z_1,\dots,z_n)=f\bigl(\us_0+(\us_1-\us_0)z_1+\cdots+(\us_n-\us_0)z_n\bigr). 
$$
Since $\{\us_1-\us_0,\dots,\us_n-\us_0\}$ is a basis of $\C^n$, the condition $f=0$ is equivalent to $\tilde{f}=0$ and the conditions \eqref{Eq:unicityCasGeneral} for $f$ are equivalent to the conditions \eqref{Eq:unicityei}
for $\tilde{f}$. 
From the assumption on the exponential type of $f$ we deduce that the function $\tilde{f}$
has exponential type  $< \pi $  in each of the variables. Hence Proposition \ref{Prop:PoritskyCasGeneral} follows from Proposition \ref{Prop:PoritskyGeneralise}.
\end{proof}

\begin{corollary}\label{Corollary:coefficientsrationnels}
Let $K$ be a field containing all coordinates of $\us_0,\us_1,\dots,\us_n$.
A polynomial  $f\in\C[\uz]$ which satisfies 
 \eqref{Eq:but} 
for  $(\ut,i)\in\calT$ belongs to $K[\uz]$.  
\end{corollary}

\begin{proof}
For $(\ut,i)\in\calT$, set $a_{\ut,i}=(D^\ut f)(\us_i)$. Since $f$ is a polynomial, the set of $(\ut,i)\in\calT$ such that $a_{\ut,i}\not=0$ is finite. 
By assumption, $a_{\ut,i}\in\Z$. 
Proposition \ref{Prop:PoritskyGeneralise} shows that $f$  is the unique polynomial satisfying $(D^\ut f)(\us_i)=a_{t,i}$ for all $(\ut,i)\in\calT$. Hence the coefficients of the polynomial $f$ are the unique solution to a system of linear equations with coefficients in $K$. Therefore these coefficients are in $K$.  
\end{proof}

\section{Proof of Theorem \ref{Th:but}}

The proof of  Theorem \ref{Th:but} will use the following easy Lemma:

\begin{lemma}
\label{Lemme:DeriveesPolynomes}
Let $f$ be an analytic function at $0$ in $\C^n$ and let $D$ be a positive integer. 
Assume that for all $\ut\in (2\N)^n$ with $\Vert \ut \Vert=D$, we have 
$$ 
D^\ut f =0.
$$  
Then  $f$ is a polynomial of   total degree $< D+n$;
\end{lemma}
 
 \begin{proof}
 Assume $f$ satisfies the assumptions of Lemma \ref{Lemme:DeriveesPolynomes}. 
 For $\unu=(\nu_1,\dots,\nu_n)\in\{0,1\}^n$, we have $D^{\ut+\unu} f=0$. Hence $D^\utau f=0$ for all $\utau\in\N^n$ satisfying $\Vert \utau \Vert=D+n$. This implies $(D^\uk f)(0)=0$ for all $\uk\in\N^n$ with $\Vert \uk \Vert\ge D+n$.
 The conclusion follows. 
 \end{proof}

 \begin{proof}
 [Proof of Theorem \ref{Th:but}]
 Assume that $f$ satisfies the assumptions of Theorem \ref{Th:but}. 
Given the growth assumption  \eqref{eq:maingrowthcondition}, Proposition \ref{Proposition:Polya} shows that the set of $(\ut,i)$ with $\ut\in\N^n$, $i=0,1,\dots,n$ and
$$
\left| (D^\ut f)(\us_i) \right|<1
$$
is finite. 
Therefore there exists an even integer $T_0$ such that, for $\Vert \ut\Vert\ge T_0$ and $0\le i\le n$ with $(\ut,i)\in\calT$,  we have $(D^\ut f)(\us_i) =0$. 

Let $\tau_1,\dots,\tau_n$ be even integers with $\Vert \utau\Vert \ge T_0$. Denote by $\hat f$ the function $D^\utau f$. For $(\ut,i)\in\calT$, we have $(\ut+\utau,i)\in\calT$ and $\Vert \ut+ \utau\Vert \ge T_0$, hence 
$(D^\ut \hat f)(\us_i)=(D^{\ut+\utau}   f)(\us_i)=0$.  Assuming that the exponential type of $f$ is $< \pi$ in each direction $\us_i-\us_0$, 
we deduce the same for $\hat f$, and then Proposition \ref{Prop:PoritskyCasGeneral} implies $\hat f=0$.  Hence $D^\utau f=0$ for all $\tau_1,\dots,\tau_n$   even integers with $\Vert \utau\Vert \ge T_0$.  It follows from Lemma \ref{Lemme:DeriveesPolynomes} that $f$ is a polynomial of total degree $<  T_0 +n$. 
\end{proof}

\end{document}